\documentclass{amsart}

\usepackage{geometry}
\usepackage{amssymb}
\usepackage{verbatim}
\usepackage{graphicx}
\usepackage{color}

\newtheorem{prop}{Proposition}[section]
\newtheorem{teo}[prop]{Theorem}

\newtheorem{lm}[prop]{Lemma}
\theoremstyle{definition}
\newtheorem{defi}[prop]{Definition}
\newtheorem{oss}[prop]{Remark}
\newtheorem*{ack}{Acknowledgments}

\usepackage[colorlinks=true,urlcolor=blue, citecolor=red,linkcolor=blue,
linktocpage,pdfpagelabels, bookmarksnumbered,bookmarksopen]{hyperref}
\usepackage[hyperpageref]{backref}

\date{\today}

\numberwithin{equation}{section}

\title[The supersolution method for Hardy's inequality]{A note on the\\ supersolution method for Hardy's inequality}

\author[Bianchi]{Francesca Bianchi}
\address[F.\ Bianchi]{Dipartimento di Scienze Matematiche, Fisiche e Informatiche
	\newline\indent
	Universit\`a di Parma
	\newline\indent
	Parco Area delle Scienze 53/a, Campus, 43124 Parma, Italy}
\email{francesca.bianchi@unipr.it}

\author[Brasco]{Lorenzo Brasco}
\address[L.\ Brasco]{Dipartimento di Matematica e Informatica
	\newline\indent
	Universit\`a degli Studi di Ferrara
	\newline\indent
	Via Machiavelli 35, 44121 Ferrara, Italy}
\email{lorenzo.brasco@unife.it}

\author[Sk]{Firoj Sk}
\address[F. Sk]{Department of Mathematics
\newline\indent
Indian Institute of Science Education and Research
\newline\indent 
Dr. Homi Bhaba road, 411008 Pune, India}
\email{firojmaciitk7@gmail.com}

\author[Zagati]{Anna Chiara Zagati}
\address[A.\ C.\ Zagati]{Dipartimento di Scienze Matematiche, Fisiche e Informatiche
	\newline\indent
	Universit\`a di Parma
	\newline\indent
	Parco Area delle Scienze 53/a, Campus, 43124 Parma, Italy}
\email{annachiara.zagati@unipr.it}

\begin{document}

\begin{abstract}
We prove a characterization of Hardy's inequality in Sobolev-Slobodecki\u{\i} spaces in terms of positive local weak supersolutions of the relevant Euler-Lagrange equation. This extends previous results by Ancona and Kinnunen \& Korte for standard Sobolev spaces. The proof is based on variational methods.
\end{abstract}

\maketitle

\begin{center}
\begin{minipage}{10cm}
\small
\tableofcontents
\end{minipage}
\end{center}

\section{Introduction}
 
 \subsection{Main result}
The present note deals with the fractional Hardy inequality in Sobolev-Slobodecki\u{\i} spaces. For $1<p<\infty$, $0<s<1$ and $\Omega\subsetneq\mathbb{R}^N$ an open set, this takes the form
\begin{equation}
\label{hardy}
C\,\int_\Omega \frac{|u|^p}{d_\Omega^{s\,p}}\,dx\le \iint_{\mathbb{R}^N\times\mathbb{R}^N} \frac{|u(x)-u(y)|^p}{|x-y|^{N+s\,p}}\,dx\,dy,\qquad \mbox{ for every } u\in C^\infty_0(\Omega),
\end{equation}
where 
\[
d_\Omega(x)=\min_{y\in\partial\Omega} |x-y|,\qquad \mbox{ for }x\in\Omega.
\]
This note can be seen as a companion paper of \cite{BBZ}, where the main result here presented is applied. We refer the reader to \cite{BBZ} for more details on \eqref{hardy}.
The paper \cite{BBZ} deals with the problem of determining the sharp constant in \eqref{hardy}.
This is the quantity defined by 
\[
\mathfrak{h}_{s,p}(\Omega):=\inf_{u\in C^\infty_0(\Omega)}\left\{\iint_{\mathbb{R}^N\times\mathbb{R}^N} \frac{|u(x)-u(y)|^p}{|x-y|^{N+s\,p}}\,dx\,dy\, :\, \int_\Omega \frac{|u|^p}{d_\Omega^{s\,p}}\,dx=1\right\}.
\]
In order to present the main result, we need to introduce the equation
\begin{equation}
\label{equazione}
(-\Delta_p)^s u=\lambda\,\frac{|u|^{p-2}\,u}{d_\Omega^{s\,p}},\qquad \mbox{ in }\Omega,
\end{equation}
where $\lambda\ge 0$. The symbol $(-\Delta_p)^s$ stands for the {\it fractional $p-$Laplacian of order $s$}, defined in weak form by the first variation of the convex functional
\[
u\mapsto \frac{1}{p}\,[u]_{W^{s,p}(\mathbb{R}^N)}^p:=\frac{1}{p}\,\iint_{\mathbb{R}^N\times\mathbb{R}^N} \frac{|u(x)-u(y)|^p}{|x-y|^{N+s\,p}}\,dx\,dy,
\]
see Section \ref{sec:2} for more details.
\par
Actually, there is tight connection between {\it positive supersolutions} of \eqref{equazione} and the constant $\mathfrak{h}_{s,p}(\Omega)$. This connection is encoded in the following result, which is the main goal of the present note. We refer to the comments below, for a comparison with known results.
\begin{teo}
	\label{teo:duale}
	Let $1<p<\infty$, $0<s<1$ and let $\Omega\subsetneq\mathbb{R}^N$ be an open set. Then we have 
	\[
	\mathfrak{h}_{s,p}(\Omega)=\sup\Big\{\lambda\ge 0\, :\, \mbox{equation \eqref{equazione} admits a positive local weak supersolution}\Big\}.
	\]
\end{teo}
The proof of this result is a direct consequence of the following equivalence (see Lemmas \ref{lm:2} and \ref{lm:1} below)
\[
\mathfrak{h}_{s,p}(\Omega)>0\qquad \Longleftrightarrow \qquad \begin{array}{c}\mbox{ the equation \eqref{equazione} admits a positive}\\
	\mbox{ local weak supersolution for some $\lambda>0$}.
\end{array}
\]
Such an equivalence is quite well-known among experts, at least in the local case, i.e. for standard Sobolev spaces. The case $p=2$ is contained for example in the classical paper \cite[Appendix]{An} by Alano Ancona and it has been generalized to $p\not=2$ by Juha Kinnunen and Riika Korte, see \cite[Theorem 5.1]{KK}. 
We should notice that the case $p=2$ and $0<s<1$ can be obtained from \cite[Theorem 1.9]{Fitz} by Patrick Joseph Fitzsimmons, which is concerned with the more general framework of Dirichlet forms associated to symmetric Markov processes. We point out that in \cite{Fitz} the author uses a probabilistic approach, which can not be applied to the case $p\not=2$.
\par
In all these references, the proof of the implication ``$\Longrightarrow$'' is based on the Lax-Milgram Theorem for coercive bilinear forms and its non-Hilbertian variants\footnote{As an historical curiosity, we notice that the prototype of this kind of result can be traced back to a paper by Giuseppe Tomaselli, dealing with weighted Hardy inequalities for standard Sobolev spaces, in the one-dimensional case. We refer to \cite[Lemma 2, point (ii)]{To} for more details.}. Our proof of this fact is slightly different: more precisely, we use a purely variational approach (see Lemma \ref{lm:1} below) in order to show existence of a supersolution. This is quite delicate, since in Theorem \ref{teo:duale} we do not have any assumption on the open set $\Omega$, apart from the fact that it admits a Hardy inequality. Thus, when employing the Direct Method in the Calculus of Variations, some non-trivial compactness issues arise. A careful study of a suitable weighted Sobolev-Slobodecki\u{\i} space is needed at that point (see Section \ref{sec:3}). This also gives us the opportunity to make some precisions on the correct functional analytic setting which is needed for this result (see Remark \ref{oss:regolamento}).

\subsection{Plan of the paper} In Section \ref{sec:2} we introduce the main notation and definitions. Section \ref{sec:3} is devoted to discuss in detail a weighted Sobolev-Slobodecki\u{\i} space. This in an essential tool in the proof of Theorem \ref{teo:duale}, which can be found in Section \ref{sec:4}.

\begin{ack}
We wish to thank Pierre Bousquet, Juha Kinnunen and Michele Miranda for some useful discussions on the topics of this paper. 
\par
Part of this paper has been written during a staying of L.\,B. at the Institute of Applied Mathematics and Mechanics of the University of Warsaw, in the context of the {\it Thematic Research Programme: ``Anisotropic and Inhomogeneous Phenomena''}, in July 2022. Iwona Chlebicka and Anna Zatorska-Goldstein are gratefully acknowledged for their kind invitation and the nice working atmosphere provided during the whole staying.
\par
F.\,B. and A.\,C.\,Z. are members of the {\it Gruppo Nazionale per l'Analisi Matematica, la Probabilit\`a e le loro Applicazioni} (GNAMPA) of the Istituto Nazionale di Alta Matematica (INdAM). F.\,B., L.\,B. and A.\,C.\,Z.  gratefully acknowledge the financial support of the projects FAR 2019 and FAR 2020 of the University of Ferrara. F.\,S. is supported by the SERB WEA grant no.
WEA/2020/000005.
\end{ack}

\section{Notation and definitions}
\label{sec:2}
For every $1<p<\infty$, we indicate by $J_p:\mathbb{R}\to\mathbb{R}$ the monotone increasing continuous function defined by
\[
J_p(t)=|t|^{p-2}\,t,\qquad \mbox{ for every } t\in\mathbb{R}.
\]
For $x_0\in \mathbb{R}^N$ and $R>0$, we will set
\[
B_R(x_0)=\Big\{x\in\mathbb{R}^N\, :\, |x-x_0|<R\Big\},
\]
and $\omega_N:=|B_1(x_0)|$. For an open set $\Omega\subsetneq \mathbb{R}^N$, we denote by
\[
d_\Omega(x)=\min_{y\in\partial\Omega} |x-y|,\qquad \mbox{ for every }x\in\Omega,
\]
the distance function from the boundary. Whenever such a distance is bounded, we will set 
\begin{equation}
\label{inradius}
r_\Omega:=\|d_\Omega\|_{L^\infty(\Omega)}. 
\end{equation}
This will be called {\it inradius} of the set $\Omega$.
For two open sets $\Omega'\subset\Omega \subset\mathbb{R}^N$, we will write $\Omega'\Subset\Omega$ to indicate that the closure $\overline{\Omega'}$ is a compact set contained in $\Omega$. 
\vskip.2cm\noindent
For $1<p<\infty$ and $0<s<1$, we consider the fractional Sobolev space 
\[
W^{s,p}(\mathbb{R}^N)=\Big\{u\in L^p(\mathbb{R}^N)\, :\, [u]_{W^{s,p}(\mathbb{R}^N)}<+\infty\Big\},
\]
where 
\[
[u]_{W^{s,p}(\mathbb{R}^N)}:=\left(\iint_{\mathbb{R}^N\times\mathbb{R}^N} \frac{|u(x)-u(y)|^p}{|x-y|^{N+s\,p}}\,dx\,dy\right)^\frac{1}{p}.
\]
This is a reflexive Banach space, when endowed with the natural norm
\[
\|u\|_{W^{s,p}(\mathbb{R}^N)}=\|u\|_{L^p(\mathbb{R}^N)}+[u]_{W^{s,p}(\mathbb{R}^N)},
\]
see for example \cite[Proposition 17.6 \& Theorem 17.41]{Leo}.
For an open set $\Omega\subset\mathbb{R}^N$, we indicate by $\widetilde{W}^{s,p}_0(\Omega)$ the closure of $C^\infty_0(\Omega)$ in $W^{s,p}(\mathbb{R}^N)$. 
\par
Occasionally, for an open set $\Omega\subset\mathbb{R}^N$, we will need the fractional Sobolev space defined by
\[
W^{s,p}(\Omega)=\Big\{u\in L^p(\Omega)\, :\, [u]_{W^{s,p}(\Omega)}<+\infty\Big\},
\]
where 
\[
[u]_{W^{s,p}(\Omega)}:=\left(\iint_{\Omega \times \Omega} \frac{|u(x)-u(y)|^p}{|x-y|^{N+s\,p}}\,dx\,dy\right)^\frac{1}{p}.
\]
Finally, $W^{s,p}_{\rm loc}(\Omega)$ is the space of functions $u\in L^p_{\rm loc}(\Omega)$ such that $u\in W^{s,p}(\Omega')$ for every $\Omega'\Subset \Omega$.
\par
For $0<\beta<\infty$, we also denote by $L^\beta_{s\,p}(\mathbb{R}^N)$ the following weighted Lebesgue space
\[
L^\beta_{s\,p}(\mathbb{R}^N)=\left\{u\in L^\beta_{\rm loc}(\mathbb{R}^N)\, :\, \int_{\mathbb{R}^N} \frac{|u(x)|^\beta}{(1+|x|)^{N+s\,p}}\,dx<+\infty\right\}.
\]
We observe that this is a Banach space for $\beta\ge 1$, when endowed with the natural norm. Moreover, it is not difficult to see that 
\begin{equation}
\label{monoLpesi}
L^\beta_{s\,p}(\mathbb{R}^N)\subset L^{\beta'}_{s\,p}(\mathbb{R}^N),\qquad \mbox{ for } 0<\beta'<\beta<\infty.
\end{equation}
It is sufficient to use that 
\[
\int_{\mathbb{R}^N} \frac{1}{(1+|x|)^{N+s\,p}}\,dx<+\infty,\qquad \mbox{ for every }N\ge 1,\ 1<p<\infty\ \mbox{ and } 0<s<1,
\]
and then apply Jensen's inequality.
\begin{defi}
\label{defi:subsuper}
We say that $u\in W^{s,p}_{\rm loc}(\Omega)\cap L^{p-1}_{s\,p}(\mathbb{R}^N)$ is a 
\begin{itemize}
\item {\it local weak supersolution} of \eqref{equazione} if 
\begin{equation}
\label{weaksuper}
\iint_{\mathbb{R}^N\times \mathbb{R}^N} \frac{J_p(u(x)-u(y))\,(\varphi(x)-\varphi(y))}{|x-y|^{N+s\,p}}\,dx\,dy\ge \lambda\,\int_\Omega \frac{|u(x)|^{p-2}\,u(x)}{d_\Omega(x)^{s\,p}}\,\varphi(x)\,dx,
\end{equation}
for every non-negative $\varphi\in W^{s,p}(\mathbb{R}^N)$ with compact support in $\Omega$;
\vskip.2cm
\item {\it local weak solution} of \eqref{equazione} if \eqref{weaksuper} holds as an equality, for every $\varphi\in W^{s,p}(\mathbb{R}^N)$ with compact support in $\Omega$.
\end{itemize}
\end{defi}
\begin{oss}
Under the assumptions taken on $u$ and the test functions, the previous definition is well-posed, i.e.
\[
 \frac{J_p(u(x)-u(y))\,(\varphi(x)-\varphi(y))}{|x-y|^{N+s\,p}}\in L^1(\mathbb{R}^N\times\mathbb{R}^N).
\]
We also observe that if a local weak solution $u$ belongs to $\widetilde W^{s,p}_0(\Omega)$, then by a density argument we can take $\varphi=u$ itself as a test function in the weak formulation.
\end{oss}

\section{A weighted fractional Sobolev space}
\label{sec:3}
In the proof of Theorem \ref{teo:duale}, we will crucially exploit a suitable weighted fractional Sobolev space, whose definition is inspired by \cite[Appendix]{An}.
\begin{defi}\label{def:X}
Let $1<p<\infty$, $0<s<1$ and let $\Omega\subsetneq\mathbb{R}^N$ be an open set, we define
\[ 
\mathcal{X}^{s,p}(\Omega; d_{\Omega}) := \left\{ u \in L^{p}_{s\,p}(\mathbb{R}^N) : [u]_{W^{s,p}(\mathbb{R}^N)} < +\infty\ \mbox{ and }\ \frac{u}{d_{\Omega}^s} \in L^p(\Omega)\right\}, 
\]
endowed with the norm
\[ 
\| u \|_{\mathcal{X}^{s,p}(\Omega; d_{\Omega})}:= [u]_{W^{s,p}(\mathbb{R}^N)} + \left(\int_{\Omega} \frac{|u|^p}{d_{\Omega}^{s\,p}} \,dx \right)^\frac{1}{p},\qquad \mbox{ for every } u\in \mathcal{X}^{s,p}(\Omega; d_{\Omega}).
\]
Then we define $\mathcal{X}^{s,p}_0(\Omega; d_{\Omega})$ as the closure of 
$C^{\infty}_0(\Omega)$ in $\mathcal{X}^{s,p}(\Omega;d_\Omega)$. 
\end{defi}
\begin{oss}
\label{oss:estendi}
We observe that if the open set $\Omega\subsetneq\mathbb{R}^N$ is such that $\mathfrak{h}_{s,p}(\Omega)>0$, then by a simple density argument we can assure that Hardy's inequality holds in $\mathcal{X}^{s,p}_0(\Omega;d_\Omega)$, as well. That is, we have
\[
\mathfrak{h}_{s,p}(\Omega)\,\int_{\Omega} \frac{|u|^p}{d_{\Omega}^{s\,p}} \,dx \le [u]^p_{W^{s,p}(\mathbb{R}^N)},\qquad \mbox{ for every } u\in \mathcal{X}^{s,p}_0(\Omega;d_\Omega).
\]
Accordingly, this implies that in this case
\[
u\mapsto [u]_{W^{s,p}(\mathbb{R}^N)},
\]
defines an equivalent norm on $\mathcal{X}^{s,p}_0(\Omega;d_\Omega)$.
\end{oss}

\begin{prop}
\label{lm:banach}
Let $1<p<\infty$ and $0<s<1$. Let $\Omega\subsetneq\mathbb{R}^N$ be an open set. Then
\[
\mathcal{X}^{s,p}(\Omega;d_\Omega)\subset W^{s,p}_{\rm loc}(\Omega)\cap L^{p-1}_{s\,p}(\mathbb{R}^N),
\]
and we have the estimate
\begin{equation}
\label{boh!}
\left(\int_{\mathbb{R}^N} \frac{|u|^p}{(1+|x|)^{N+s\,p}}\,dx\right)^\frac{1}{p}\le C_\Omega\,\| u \|_{\mathcal{X}^{s,p}(\Omega;d_\Omega)},\qquad \mbox{ for every } u\in \mathcal{X}^{s,p}(\Omega;d_\Omega).
\end{equation}
Moreover, $\mathcal{X}^{s,p}(\Omega;d_\Omega)$ and $\mathcal{X}^{s,p}_0(\Omega;d_\Omega)$ are Banach spaces.
\end{prop}
\begin{proof}
The first fact is straightforward, by also taking into account \eqref{monoLpesi}.
\par
We prove the estimate \eqref{boh!}. We take a ball $B_{R}(x_0)\Subset \Omega$ such that $B_{2\,R}(x_0)\Subset\Omega$, as well. We then write
\[
\begin{split}
[u]_{W^{s,p}(\mathbb{R}^N)}&=\left(\iint_{\mathbb{R}^N\times \mathbb{R}^N} \frac{|u(x)-u(y)|^p}{|x-y|^{N+s\,p}}\,dx\,dy\right)^\frac{1}{p}\\
&\ge \left(\iint_{B_{R}(x_0)\times (\mathbb{R}^N\setminus B_{2\,R}(x_0))} \frac{|u(x)-u(y)|^p}{|x-y|^{N+s\,p}}\,dx\,dy\right)^\frac{1}{p}\\
&\ge \left(\int_{B_{R}(x_0)}\left(\int_{\mathbb{R}^N\setminus B_{2\,R}(x_0)} \frac{|u(y)|^p}{|x-y|^{N+s\,p}}\,dy\right)\,dx\right)^\frac{1}{p}\\
&-\left(\int_{B_{R}(x_0)}|u(x)|^p\,\left(\int_{\mathbb{R}^N\setminus B_{2\,R}(x_0)} \frac{1}{|x-y|^{N+s\,p}}\,dy\right)\,dx\right)^\frac{1}{p},
\end{split}
\]
thanks to Minkowski's inequality.
By observing that 
\[
|x-y|\ge \frac{1}{2}\,|y-x_0|,\qquad \mbox{ for every } x\in B_R(x_0), \ y\not\in B_{2\,R}(x_0),
\]
we have 
\[
\begin{split}
\int_{B_{R}(x_0)}|u(x)|^p\,\left(\int_{\mathbb{R}^N\setminus B_{2\,R}(x_0)} \frac{1}{|x-y|^{N+s\,p}}\,dy\right)\,dx&\le \frac{N\,\omega_N\,2^N}{s\,p}\,R^{-s\,p}\,\int_{B_R(x_0)} |u|^p\,dx\\
&\le \frac{N\,\omega_N\,2^N}{s\,p}\,R^{-s\,p}\,\int_{\Omega} \frac{|u|^p}{d_\Omega^{s\,p}}\,dx\,\|d_\Omega\|^{s\,p}_{L^\infty(B_R(x_0))}.
\end{split}
\]
This implies that we have 
\begin{equation}
\label{mezzanotte}
\left(\int_{B_{R}(x_0)}\left(\int_{\mathbb{R}^N\setminus B_{2\,R}(x_0)} \frac{|u(y)|^p}{|x-y|^{N+s\,p}}\,dy\right)\,dx\right)^\frac{1}{p}\le C\,\| u \|_{\mathcal{X}^{s,p}(\Omega; d_{\Omega})},
\end{equation}
for a constant $C=C(N,s,p,\Omega,B_R(x_0))>0$. We now use that 
\[
|x-y|\le 2\,|x_0-y|,\qquad \mbox{ for every } x\in B_R(x_0), \ y\not\in B_{2\,R}(x_0),
\]
together with the fact that 
\[
|x_0-y|\le |x_0|+|y|\le (1+|x_0|)\,(1+|y|).
\]
By using these in \eqref{mezzanotte}, we get
\[
\left(\int_{\mathbb{R}^N\setminus B_{2\,R}(x_0)} \frac{|u(y)|^p}{(1+|y|)^{N+s\,p}}\,dy\right)^\frac{1}{p}\le C\,\| u \|_{\mathcal{X}^{s,p}(\Omega; d_{\Omega})},
\]
possibly with a different constant $C=C(N,s,p,\Omega,B_R(x_0))>0$. The proof of estimate \eqref{boh!} is almost over: it is now sufficient to add on both sides the term
\[
\left(\int_{B_{2\,R}(x_0)} \frac{|u(y)|^p}{(1+|y|)^{N+s\,p}}\,dy\right)^\frac{1}{p}.
\]
Then by using that 
\[
\begin{split}
\int_{B_{2\,R}(x_0)} \frac{|u(y)|^p}{(1+|y|)^{N+s\,p}}\,dy\le \int_{B_{2\,R}(x_0)} |u(y)|^p\,dy&\le \int_{\Omega} \frac{|u|^p}{d_\Omega^{s\,p}}\,dy\,\|d_\Omega\|^{s\,p}_{L^\infty(B_{2\,R}(x_0))}\\
&\le C_\Omega\,\|u\|^p_{\mathcal{X}^{s,p}(\Omega;d_\Omega)},
\end{split}
\]
we eventually get the desired conclusion. 
\vskip.2cm\noindent
We prove the second part of the statement. We first observe that it is sufficient to prove that $\mathcal{X}^{s,p}(\Omega;d_\Omega)$ is a Banach space. We take $\{u_n\}_{n\in\mathbb{N}}\subset \mathcal{X}^{s,p}(\Omega;d_\Omega)$ to be a Cauchy sequence. Then we get that this is a Cauchy sequence in the Banach space $L^p(\Omega;d_\Omega^{-s\,p})$ and that 
\[
\{D^s u_n\}_{n\in\mathbb{N}}\qquad \mbox{ where } D^s \varphi(x,y):=\frac{\varphi(x)-\varphi(y)}{|x-y|^{\frac{N}{p}+s}},
\]
is a Cauchy sequence in $L^p(\mathbb{R}^N\times\mathbb{R}^N)$. This follows from the fact that
\[
[u_n]_{W^{s,p}(\mathbb{R}^N)}=\|D^s u_n\|_{L^p(\mathbb{R}^N\times\mathbb{R}^N)}.
\]
Moreover, according to \eqref{boh!}, the sequence $\{u_n\}_{n\in\mathbb{N}}$ is also a Cauchy sequence in the Banach space $L^p_{s\,p}(\mathbb{R}^N)$. The last fact implies that there exists $u\in L^p_{s\,p}(\mathbb{R}^N)$ such that 
\[
\lim_{n\to\infty}\int_{\mathbb{R}^N} \frac{|u_n-u|^p}{(1+|x|)^{N+s\,p}}\,dx=0.
\]
In particular, up to a subsequence, we can suppose that $u_n$ converges to $u$ almost everywhere in $\mathbb{R}^N$. By using the completeness of $L^p(\Omega;d_\Omega^{-s\,p})$, we get similarly the existence of $\widetilde{u}\in L^p(\Omega;d_\Omega^{-s\,p})$ such that 
\[
\lim_{n\to\infty} \int_\Omega \frac{|u_n-\widetilde u|^p}{d_\Omega^{s\,p}}\,dx=0.
\]
By uniqueness of the limit, we must have $u=\widetilde u$ almost everywhere in $\Omega$. Finally, by using that $L^p(\mathbb{R}^N\times\mathbb{R}^N)$ is a Banach space, we get that there exists $\phi\in L^p(\mathbb{R}^N\times\mathbb{R}^N)$ such that
\[
\lim_{n\to\infty} \|D^s u_n-\phi\|_{L^p(\mathbb{R}^N\times\mathbb{R}^N)}=0.
\]
This in particular would imply that 
\[
\lim_{n\to\infty}D^s u_n(x,y)=\phi(x,y),\qquad \mbox{ for a.\,e. } (x,y)\in\mathbb{R}^N\times\mathbb{R}^N,
\]
up to a subsequence. On the other hand, by using the almost everywhere convergence of $u_n$ previously inferred, we also obtain that 
\[
\lim_{n\to\infty}D^s u_n(x,y)=D^s u(x,y),\qquad \mbox{ for a.\,e. } (x,y)\in\mathbb{R}^N\times\mathbb{R}^N.
\]
By using the uniqueness of the limit, we get at the same time that 
\[
[u]_{W^{s,p}(\mathbb{R}^N)}<+\infty \qquad \mbox{ and }\qquad  \lim_{n\to\infty} \|D^s u_n-D^s u\|_{L^p(\mathbb{R}^N\times\mathbb{R}^N)}=\lim_{n\to\infty} [u_n-u]_{W^{s,p}(\mathbb{R}^N)}=0.
\]
This concludes the proof.
\end{proof}

In the next technical lemma, we show that the summability of a {\it negative} power of the distance implies certain geometric properties of the open set.
\begin{lm}\label{lm:sumdist}
Let $N\ge 1$ and let $\Omega\subsetneq\mathbb{R}^N$ be an open set such that 
\[
\int_\Omega \frac{1}{d_\Omega^\alpha}\,dx<+\infty,
\]
for some $\alpha>0$. Then we must have $\alpha<N$. Moreover, we have the estimates 
\begin{equation}
\label{stimegeo}
r_\Omega\le \left(\frac{2^\alpha}{\omega_N}\,\int_\Omega \frac{1}{d_\Omega^\alpha}\,dx\right)^\frac{1}{N-\alpha}\qquad \mbox{ and }\qquad |\Omega|\le \left(\frac{2^\alpha}{\omega_N}\right)^\frac{\alpha}{N-\alpha}\,\left(\int_\Omega \frac{1}{d_\Omega^\alpha}\,dx\right)^\frac{N}{N-\alpha},
\end{equation}
where $r_\Omega$ is defined in \eqref{inradius}.
\end{lm}
\begin{proof}
We take $x_0\in\Omega$ and consider the open ball $B_r(x_0)$ with radius $r=d_\Omega(x_0)$. This implies that 
\[
B_r(x_0)\subset\Omega\qquad \mbox{ and }\qquad \partial B_r(x_0)\cap \partial \Omega\not=\emptyset. 
\]
Let us call $\widetilde{x}_0$ a point contained in this intersection. By observing that 
\[
d_\Omega(x)\le |x-\widetilde{x}_0|,\qquad \mbox{ for every } x\in B_{r}(x_0),
\]
we get
\[
+\infty>\int_\Omega \frac{1}{d_\Omega^\alpha}\,dx\ge \int_{B_r(x_0)} \frac{1}{|x-\widetilde{x}_0|^\alpha}\,dx.
\]
By using spherical coordinates, we see that the last integral diverges for $\alpha\ge N$. Thus we get the first statement.
\par
In order to get the claimed estimates, we go on by estimating from below the last integral as follows
\[
+\infty>\int_\Omega \frac{1}{d_\Omega^\alpha}\,dx\ge \frac{1}{2^\alpha\,r^\alpha}\,\int_{B_{r}(x_0)} \,dx=\frac{\omega_N}{2^\alpha}\,r^{N-\alpha}=\frac{\omega_N}{2^\alpha}\,d_\Omega(x_0)^{N-\alpha}.
\]
Since $\alpha<N$ from the first part of the proof, we can take the supremum on $x_0\in \Omega$ and get that the distance function is actually bounded. Moreover, we obtain the first estimate in \eqref{stimegeo}, thus in particular the inradius is finite. In turn, by using this fact we get
\[
\int_\Omega\frac{1}{d_\Omega^\alpha}\,dx\ge \frac{1}{r_\Omega^\alpha} \int_{\Omega}\,dx =\frac{|\Omega|}{r_\Omega^\alpha},
\]
which shows that the volume is finite, as well, together with the second estimate in \eqref{stimegeo}.
This concludes the proof.
\end{proof}
As a straightforward consequence of Lemma \ref{lm:sumdist}, we get the following
\begin{lm}
\label{lm:sfiga}
Let $1<p<\infty$, $0<s<1$ and let $\Omega\subsetneq\mathbb{R}^N$ be an open set. 
Then for $s\,p\ge N$ the unique constant function contained in $\mathcal{X}^{s,p}(\Omega;d_\Omega)$ is the null one. 
\par
The same conclusion holds also for $s\,p<N$, if we additionally suppose that $|\Omega|=+\infty$.
\end{lm}
In the next result we compare the two spaces $\widetilde{W}^{s,p}_0(\Omega)$ and $\mathcal{X}^{s,p}_0(\Omega;d_\Omega)$.
\begin{prop}
\label{prop:fortuna}
Let $1<p<\infty$ and $0<s<1$. Let $\Omega\subsetneq\mathbb{R}^N$ be an open set such that $\mathfrak{h}_{s,p}(\Omega)>0$.
Then we have 
\begin{equation}
\label{inclusion}
\widetilde{W}^{s,p}_0(\Omega)\subseteq \mathcal{X}^{s,p}_0(\Omega;d_{\Omega}),
\end{equation}
and the inclusion is continuous. Moreover, if we assume that $r_\Omega<+\infty$, then
\[
\widetilde{W}^{s,p}_0(\Omega)= \mathcal{X}^{s,p}_0(\Omega;d_{\Omega}),
\]
and 
\begin{equation}
\label{normaX}
\varphi\mapsto [\varphi]_{W^{s,p}(\mathbb{R}^N)},
\end{equation}
is an equivalent norm on this space.
Finally, if we further require that $|\Omega|<+\infty$, then we have the continuous embedding 
\[
\mathcal{X}^{s,p}_0(\Omega; d_{\Omega})\hookrightarrow L^p(\Omega),
\]
and this is compact, as well.
\end{prop}
\begin{proof} By recalling Remark \ref{oss:estendi}, we know that \eqref{normaX} is an equivalent norm on $\mathcal{X}^{s,p}_0(\Omega;d_\Omega)$. Since we trivially have 
\[
[\varphi]_{W^{s,p}(\mathbb{R}^N)} \le \|\varphi\|_{W^{s,p}(\mathbb{R}^N)},\qquad \mbox{ for every } \varphi\in C^\infty_0(\Omega),
\]
the continuous inclusion \eqref{inclusion} easily follows. 
\par
We now assume that $r_\Omega<+\infty$. In conjuction with Hardy's inequality and recalling \eqref{inradius}, this yields
\[
\int_\Omega |\varphi|^p\,dx\le r_\Omega^{s\,p}\,\int_\Omega \frac{|\varphi|^p}{d_\Omega^{s\,p}}\,dx\le \frac{r_\Omega^{s\,p}}{\mathfrak{h}_{s,p}(\Omega)}\,[\varphi]^p_{W^{s,p}(\mathbb{R}^N)},\qquad \mbox{ for every } \varphi\in C^\infty_0(\Omega). 
\]
Thus we get that 
\[
\varphi\mapsto \|\varphi\|_{\mathcal{X}^{s,p}(\mathbb{R}^N)}\qquad \mbox{ and }\qquad \varphi\mapsto \|\varphi\|_{W^{s,p}(\mathbb{R}^N)},
\]
are equivalent norms on $C^\infty_0(\Omega)$, again thanks to Remark \ref{oss:estendi}. Then the claimed identity of the two closures immediately follows. The last statement is now an easy consequence of the same property for the space $\widetilde{W}^{s,p}_0(\Omega)$, which is well-known.
\end{proof}
\begin{oss}
\label{oss:regolamento}
Under the sole assumption that $\mathfrak{h}_{s,p}(\Omega)>0$, in general we have
\[
\widetilde{W}^{s,p}_0(\Omega) \subset \mathcal{X}^{s,p}_0(\Omega; d_{\Omega})\qquad \mbox{ and }\qquad \widetilde{W}^{s,p}_0(\Omega) \not= \mathcal{X}^{s,p}_0(\Omega; d_{\Omega}),
\]
contrary to what incorrectly stated in \cite[Theorem 5.1]{KK}, for the local case $s=1$. 
As a counter-example, it is sufficient to take any open set $\Omega\subsetneq\mathbb{R}^N$ such that 
\[
\mathfrak{h}_{s,p}(\Omega)>0\qquad \mbox{ and }\qquad \inf_{\varphi\in C^\infty_0(\Omega)} \left\{[\varphi]^p_{W^{s,p}(\mathbb{R}^N)}\, :\, \int_\Omega |\varphi|^p\,dx=1\right\}=0.
\]
For example, we can take $\Omega$ to be a half-space. In such a case, we have by construction
\[
\widetilde{W}^{s,p}_0(\Omega)\hookrightarrow L^p(\Omega),
\]
while
\[
\mathcal{X}^{s,p}_0(\Omega;d_\Omega)\not\hookrightarrow L^p(\Omega).
\]
\end{oss}

We can finally prove a compactness result for the space $\mathcal{X}^{s,p}_0(\Omega)$, under minimal assumptions on the open set $\Omega$. This will be crucially exploited in the proof of Lemma \ref{lm:1}.
\begin{teo}
\label{lm:X}
Let $1<p<\infty$, $0<s<1$ and let $\Omega\subsetneq\mathbb{R}^N$ be an open set such that $\mathfrak{h}_{s,p}(\Omega)>0$.
Let $\{u_n\}_{n \in \mathbb{N}}\subset \mathcal{X}^{s,p}_0(\Omega;d_\Omega)$ be such that
\[
[u_n]^p_{W^{s,p}(\mathbb{R}^N)}\le M,\qquad \mbox{ for every } n\in\mathbb{N},
\]
for some $M>0$. Then there exist a function $u \in \mathcal{X}^{s,p}_0(\Omega; d_{\Omega})$ and subsequence $\{u_{n_k}\}_{k\in\mathbb{N}}$ such that
\[ 
\lim_{k\to\infty}u_{n_k}(x)= u(x), \qquad \mbox{ for a.\,e. } x\in \Omega. 
\] 
Moreover, for every $\Omega'\Subset\Omega$, we also have
\[
\lim_{k\to\infty} \|u_{n_k}-u\|_{L^p(\Omega')}=0,
\]
up to a possible further subsequence.
\end{teo}

\begin{proof}
We need two distinguish two cases: either $|\Omega|<+\infty$ or $|\Omega|=+\infty$.
\vskip.2cm\noindent
{\it Case 1: $\Omega$ has finite volume.} This is the easiest case: here the result plainly follows from Proposition \ref{prop:fortuna}. We also observe that the last statement actually holds in a stronger from, since we can infer convergence in $L^p(\Omega)$. 
\vskip.2cm\noindent
{\it Case 2: $\Omega$ has infinite volume.} We still use the notation $D^s \varphi$ for a measurable function, as in Proposition \ref{lm:banach}. Thus, by assumption, we get that $\{D^s u_n\}_{n\in\mathbb{N}}$ is a bounded sequence in $L^p(\mathbb{R}^N\times\mathbb{R}^N)$. This entails that, up to a subsequence, it is weakly converging in $L^p(\mathbb{R}^N\times \mathbb{R}^N)$. Let us call $\phi$ such a limit. We may apply Mazur's Lemma (see \cite[Theorem 2.13]{LL}) and get that for every $n\in\mathbb{N}$ there exists 
	\[
\big\{t_\ell(n)\big\}_{\ell=0}^n\subset[0,1],\qquad \mbox{ such that }
\qquad	\sum_{\ell=0}^n t_\ell(n)=1,
	\]
and such that the new sequence made of convex combinations
	\[
	\widetilde{\phi}_{n}(x,y)=\sum_{\ell=0}^n t_\ell(n)\,D^s u_\ell(x,y),
	\]
strongly converges in $L^p(\mathbb{R}^N\times\mathbb{R}^N)$ to $\phi$, as $n$ goes to $\infty$. Observe that by construction we have
\[
\sum_{\ell=0}^n t_\ell(n)\,D^s u_\ell=D^s\left(\sum_{\ell=0}^n t_\ell(n)\,u_\ell\right),
\]
and
\[
\widetilde{u}_n:=\sum_{\ell=0}^n t_\ell(n)\,u_\ell\in \mathcal{X}^{s,p}_0(\Omega;d_\Omega),
\]
since the latter is a vector space.
This proves that $\{D^s \widetilde{u}_n\}_{n\in\mathbb{N}}$ is a Cauchy sequence in $L^p(\mathbb{R}^N\times\mathbb{R}^N)$ and this, in turn, implies that $\{\widetilde u_n\}_{n\in\mathbb{N}}$ is a Cauchy sequence in $\mathcal{X}^{s,p}_0(\Omega;d_\Omega)$, thanks to Remark \ref{oss:estendi}. By using that $\mathcal{X}^{s,p}_0(\Omega;d_\Omega)$ is a Banach space, we get that $\{\widetilde u_n\}_{n\in\mathbb{N}}$ converges in this space to a limit function $u\in \mathcal{X}^{s,p}_0(\Omega;d_\Omega)$. In particular, we must have 
\[
D^s u=\phi.
\]
We now want to prove that $\{u_n\}_{n\in\mathbb{N}}$ converges almost everywhere on $\mathbb{R}^N$ to the function $u$, up to a subsequence. We first observe that all the elements of $\mathcal{X}^{s,p}_0(\Omega;d_\Omega)$ vanish almost everywhere in $\mathbb{R}^N\setminus\Omega$, by construction. Thus we only need to prove convergence almost everywhere in $\Omega$.
\par
We denote by $\{\Omega_k\}_{k \in \mathbb{N}}$ an exhausting sequence for $\Omega$, made of
bounded open subsets with smooth boundary: in other words 
\[ 
\Omega_k\Subset\Omega,\quad \Omega_k \Subset \Omega_{k+1} \mbox{ for every }k\in\mathbb{N}\quad \mbox{ and }\quad \bigcup_{k \in \mathbb{N}} \Omega_k = \Omega,
\] 
see \cite[Proposition 8.2.1]{Dan}.
We preliminary observe that, thanks to the assumption $\mathfrak{h}_{s,p}(\Omega)>0$, we have for every $k,n\in\mathbb{N}$
	\[
	\int_{\Omega_k} |u_n|^p \, dx \le \|d_\Omega\|_{L^{\infty}({\Omega_k})}^{s\,p} \int_{\Omega_k} \frac{|u_n|^p}{d_{\Omega}^{s\,p}} \, dx \le \frac{1}{\mathfrak{h}_{s,p}(\Omega)}\, \|d_\Omega\|_{L^{\infty}({\Omega_k})}^{s\,p}\, [u_n]^p_{W^{s,p}(\mathbb{R}^N)}\le C_k\,M ,
	\]
which entails that $\{u_n\}_{n \in \mathbb{N}}$ is a bounded sequence in each $W^{s,p}(\Omega_k)$. By using the compactness of the embedding $W^{s,p}({\Omega_k}) \hookrightarrow L^p({\Omega_k})$ (see for example \cite[Theorem 7.1]{DPV}) and a diagonal argument, we can obtain existence of a function $U\in W^{s,p}_{\rm loc}(\Omega)$ and of a subsequence $\{u_{n_k}\}_{k\in\mathbb{N}}$ such that
\[ 
\lim_{k\to\infty}u_{n_k}(x)= U(x), \qquad \mbox{ for a.\,e. } x\in \Omega. 
\] 
We then extend $U$ to be $0$ outside $\Omega$: by using Fatou's Lemma and the almost everywhere convergence, we get
\[
[U]^p_{W^{s,p}(\mathbb{R}^N)}\le \lim_{k\to\infty} [u_{n_k}]_{W^{s,p}(\mathbb{R}^N)}^p\le M.
\]
By further using Hardy's inequality and \eqref{boh!}, we also get
\[
\int_{\Omega} \frac{|U|^p}{d_\Omega^{s\,p}}\,dx\le \liminf_{k\to\infty}\int_{\Omega} \frac{|u_{n_k}|^p}{d_\Omega^{s\,p}}\,dx\le \frac{M}{\mathfrak{h}_{s,p}(\Omega)},
\]
and
\[
\begin{split}
\left(\int_{\mathbb{R}^N} \frac{|U|^p}{(1+|x|)^{N+s\,p}}\,dx\right)^\frac{1}{p}&\le\liminf_{k\to\infty} \left(\int_{\mathbb{R}^N} \frac{|u_{n_k}|^p}{(1+|x|)^{N+s\,p}}\,dx\right)^\frac{1}{p}\\
&\le C_\Omega\,\liminf_{k\to\infty} \left[[u_{n_k}]_{W^{s,p}(\mathbb{R}^N)} + \left(\int_{\Omega} \frac{|u_{n_k}|^p}{d_{\Omega}^{s\,p}} \,dx \right)^\frac{1}{p}\right]\le \widetilde{C}.
\end{split}
\]
This shows that 
\[
U\in \mathcal{X}^{s,p}(\Omega;d_\Omega).
\]
We now observe that from the first part of the proof, by uniqueness of the weak limit we must have 
\[
D^s u=D^s U,\qquad \mbox{ a.\,e. in } \mathbb{R}^N\times\mathbb{R}^N.
\]
By recalling the definition of $D^s$, this in turn implies that there exists a constant $c\in\mathbb{R}$ such that 
\[
u=U+c,\qquad \mbox{ a.\,e. in } \mathbb{R}^N.
\]
By using that $\mathcal{X}^{s,p}(\Omega;d_\Omega)$ is a vector space, the function constantly equal to $c$ must belong to $\mathcal{X}^{s,p}(\Omega;d_\Omega)$. In light of Lemma \ref{lm:sfiga}, we get that $c=0$ and thus the desired conclusion holds.
\end{proof}

\section{Proof of Theorem \ref{teo:duale}}
\label{sec:4}

\begin{lm}
	\label{lm:2}
	Let $1<p<\infty$, $0<s<1$ and let $\Omega\subsetneq\mathbb{R}^N$ be an open set. Then:
	\begin{itemize}
		\item [{\it (i)}] if there exists $\lambda\ge 0$ such that the equation \eqref{equazione}
		admits a positive local weak supersolution $u$, then $\lambda\le \mathfrak{h}_{s,p}(\Omega)$;
		\vskip.2cm
		\item[{\it (ii)}] in particular, if $u$ is a solution in $ \widetilde{W}^{s,p}_0(\Omega)$,
		then $\lambda=\mathfrak{h}_{s,p}(\Omega)$ and $u$ is a minimizer for $\mathfrak{h}_{s,p}(\Omega)$.
	\end{itemize} 
\end{lm}
\begin{proof}
	In order to prove {\it (i)}, for every $\eta\in C^\infty_0(\Omega)$, we test the weak formulation with 
	\[
	\varphi=\frac{|\eta|^p}{(\varepsilon+u)^{p-1}},
	\]
	where $\varepsilon>0$.
	We observe that this is a feasible test function, thanks to \cite[Lemma 2.7]{BBZ}.
	By using the {\it discrete Picone inequality} (see \cite[Lemma 2.6]{FSspace} or \cite[Proposition 4.2]{BF}), we obtain
	\[
	\begin{split}
	\lambda\,\int_\Omega \frac{u^{p-1}}{d_\Omega^{s\,p}}\,\frac{|\eta|^p}{(\varepsilon+u)^{p-1}}\,dx&\le\iint_{\mathbb{R}^N\times \mathbb{R}^N} \frac{J_p(u(x)-u(y))\,\left(\dfrac{|\eta|^{p}}{(\varepsilon+u)^{p-1}}(x)-\dfrac{|\eta|^{p}}{(\varepsilon+u)^{p-1}}(y)\right)}{|x-y|^{N+s\,p}}\,dx\,dy\\
	&\le \iint_{\mathbb{R}^N\times\mathbb{R}^N} \frac{\Big||\eta(x)|-|\eta(y)|\Big|^p}{|x-y|^{N+s\,p}}\,dx\,dy\le [\eta]_{W^{s,p}(\mathbb{R}^N)}^p.
	\end{split}
	\]
	In the last inequality we used that 
	\begin{equation}
	\label{absolutevalue}
	\big[|\eta|\big]_{W^{s,p}(\mathbb{R}^N)}^p\le [\eta]_{W^{s,p}(\mathbb{R}^N)}^p,
	\end{equation}
and the inequality is strict, unless $\eta$ has constant sign almost everywhere (see the proof of \cite[Lemma 3.2]{BBZ}).	
	By taking the limit as $\varepsilon$ goes to $0$ on the left-hand side, using that $u$ is positive on $\Omega$ and the arbitrariness of $\eta\in C^\infty_0(\Omega)$, this finally gives that $\lambda\le \mathfrak{h}_{s,p}(\Omega)$, as desired.
\vskip.2cm\noindent
In order to prove point {\it (ii)}, we observe that if $u\in \widetilde{W}^{s,p}_0(\Omega)$, we can test the weak formulation of the equation with the solution itself. This yields
	\[
	[u]_{W^{s,p}(\mathbb{R}^N)}^p=\lambda\,\int_\Omega \frac{u^p}{d_\Omega^{s\,p}}\,dx.
	\]
	On the other hand, by definition of $\mathfrak{h}_{s,p}(\Omega)$, we know that 
	\[
	\mathfrak{h}_{s,p}(\Omega)\,\int_\Omega \frac{u^p}{d_\Omega^{s\,p}}\,dx\le [u]_{W^{s,p}(\mathbb{R}^N)}^p.
	\]
This shows that $\mathfrak{h}_{s,p}(\Omega)\le \lambda$. Since the reverse inequality holds from {\it (i)}, we conclude that it must result $\lambda=\mathfrak{h}_{s,p}(\Omega)$. 
\end{proof}

In the next Lemma, we will use the weighted space $\mathcal{X}^{s,p}_0(\Omega; d_{\Omega})$ studied in Section \ref{sec:3}. 

\begin{lm}
	\label{lm:1}
	Let $1<p<\infty$, $0<s<1$ and let $\Omega\subsetneq\mathbb{R}^N$ be an open set such that 
\[
\mathfrak{h}_{s,p}(\Omega)>0.
	\]
	Then for every $0\le \lambda<\mathfrak{h}_{s,p}(\Omega)$ there exists a local positive weak supersolution $u_\lambda\in \mathcal{X}^{s,p}_0(\Omega;d_\Omega)$ of the equation \eqref{equazione}.
	More precisely, the function $u_\lambda$ is a weak solution of the equation
	\begin{equation}
	\label{sbilenca}
	(-\Delta_p)^s u=\lambda\,\frac{u^{p-1}}{d_\Omega^{s\,p}}+1_{B},\qquad \mbox{ in }\Omega,
	\end{equation}
	where $B \Subset \Omega$ is a fixed ball.
\end{lm}

\begin{proof}
	We first observe that, for every $\varphi\in \mathcal{X}^{s,p}_0(\Omega, d_{\Omega})$, we have 
	\[
\begin{split}
\int_B |\varphi| \,dx&\le |B|^{\frac{p-1}{p}} \|d_\Omega\|_{L^{\infty}(B)}^{s} \, \left(\int_\Omega \frac{|\varphi|^p}{d_\Omega^{s\,p}}\,dx\right)^{\frac{1}{p}}\le |B|^{\frac{p-1}{p}} \,\|d_\Omega\|_{L^{\infty}(B)}^{s} \, \left(\frac{1}{\mathfrak{h}_{s,p}(\Omega)}\right)^{\frac{1}{p}}\,[\varphi]_{W^{s,p}(\mathbb{R}^N)},
\end{split}
\]	
thanks to H\"older's inequality, the definition of $\mathfrak{h}_{s,p}(\Omega)$ and the fact that Hardy's inequality holds in $\mathcal{X}^{s,p}_0(\Omega;d_\Omega)$, as well (see Remark \ref{oss:estendi}).
	This shows that we have the continuous embedding $\mathcal{X}^{s,p}_0(\Omega; d_{\Omega})\hookrightarrow L^1(B)$, for every $B \Subset \Omega$ as in the statement. 
	\par
	Let $0\le \lambda<\mathfrak{h}_{s,p}(\Omega)$, we consider the functional
	\[
	\mathfrak{F}_{\lambda}(\varphi)=\frac{1}{p}\,[\varphi]^p_{W^{s,p}(\mathbb{R}^N)}-\frac{\lambda}{p}\,\int_\Omega \frac{|\varphi|^p}{d_\Omega^{s\,p}}\,dx-\int_B \varphi\,dx,\qquad \mbox{ for every } \varphi\in \mathcal{X}^{s,p}_0(\Omega; d_{\Omega}).
	\]
	We will construct the desired supersolution as a minimizer of the following problem
	\[
	m(\lambda):=\inf_{\varphi\in \mathcal{X}^{s,p}_0(\Omega; d_{\Omega})} \mathfrak{F}_{\lambda}(\varphi).
	\]
	We first notice that by Hardy's inequality we have, for every $\varphi\in \mathcal{X}^{s,p}_0(\Omega; d_{\Omega})$
	\[
	\begin{split}
	\mathfrak{F}_{\lambda}(\varphi)&\ge \frac{1}{p}\,\left(1-\frac{\lambda}{\mathfrak{h}_{s,p}(\Omega)}\right)\,[\varphi]^p_{W^{s,p}(\mathbb{R}^N)}-\int_B \varphi\,dx\\
	&\ge \frac{1}{p}\,\left(1-\frac{\lambda}{\mathfrak{h}_{s,p}(\Omega)}\right)\,[\varphi]^p_{W^{s,p}(\mathbb{R}^N)}-\frac{p-1}{p}\,\varepsilon^\frac{1}{1-p}\,\int_B d_\Omega^{\frac{s\,p}{p-1}} \,dx-\frac{\varepsilon}{p}\,\int_B \frac{|\varphi|^p}{d_\Omega^{s\,p}}\,dx\\
	&\ge \frac{1}{p}\,\left(1-\frac{\lambda}{\mathfrak{h}_{s,p}(\Omega)}\right)\,[\varphi]^p_{W^{s,p}(\mathbb{R}^N)}-\frac{p-1}{p}\,\varepsilon^\frac{1}{1-p}\,\int_B d_\Omega^{\frac{s\,p}{p-1}} \,dx-\frac{\varepsilon}{p}\,\frac{1}{\mathfrak{h}_{s,p}(\Omega)}\,[\varphi]^p_{W^{s,p}(\mathbb{R}^N)},
	\end{split}
	\]
	with $\varepsilon>0$, where we also used Young's inequality. In particular, by choosing
	\[
	\varepsilon=\frac{\mathfrak{h}_{s,p}(\Omega)-\lambda}{2},
	\]
	we can infer that
	\begin{equation}
	\label{coercive}
	\mathfrak{F}_{\lambda}(\varphi)\ge c_1\,[\varphi]^p_{W^{s,p}(\mathbb{R}^N)}-\frac{1}{C_1},\qquad \mbox{ for every } \varphi\in \mathcal{X}^{s,p}_0(\Omega; d_{\Omega}),
	\end{equation}
where $c_1>0$ and $C_1>0$ do not depend on $\varphi$. This in particular shows that $m(\lambda)>-\infty$. 
	
	\par
	Let us now take a minimizing sequence $\{u_n\}_{n\in\mathbb{N}}\subset \mathcal{X}^{s,p}_0(\Omega; d_{\Omega})$ such that 
	\[
	\mathfrak{F}_{\lambda}(u_n)\le m(\lambda)+\frac{1}{n+1},\qquad \mbox{ for every } n\in\mathbb{N}.
	\]
	By appealing to \eqref{coercive}, we get in particular that there exists a constant $M>0$ such that
\[
	[u_n]^p_{W^{s,p}(\mathbb{R}^N)}\le M,\qquad \mbox{ for every } n\in\mathbb{N}.
\]
By applying Theorem \ref{lm:X}, we can infer existence of $u \in \mathcal{X}^{s,p}_0(\Omega; d_{\Omega})$ such that the sequence converges almost everywhere in $\mathbb{R}^N$ and such that 
\[
	\int_B u_n\,dx=\int_B u\,dx+o(1),\qquad \mbox{ as }n\to \infty,
	\]
 up to a subsequence. Observe that by construction we have 
	\[
	m(\lambda)+\frac{1}{n+1}\ge \frac{1}{p}\,[u_n]^p_{W^{s,p}(\mathbb{R}^N)}-\frac{\lambda}{p}\,\int_\Omega \frac{|u_n|^p}{d_\Omega^{s\,p}}\,dx-\int_B u_n\,dx\ge m(\lambda),
	\]
	which in particular implies that
	\begin{equation}
	\label{minimizzante}
	\frac{1}{p}\,[u_n]^p_{W^{s,p}(\mathbb{R}^N)}-\frac{\lambda}{p}\,\int_\Omega \frac{|u_n|^p}{d_\Omega^{s\,p}}\,dx-\int_B u_n\,dx=m(\lambda)+o(1),\qquad \mbox{ as }n\to \infty.
	\end{equation}
By applying the Br\'ezis-Lieb Lemma (see \cite[Theorem 1]{BL} and also \cite[Lemma 2.2]{BSY}), we get
	\[
	\frac{\lambda}{p}\,\int_\Omega \frac{|u_n|^p}{d_\Omega^{s\,p}}\,dx=\frac{\lambda}{p}\,\int_\Omega \frac{|u|^p}{d_\Omega^{s\,p}}\,dx+\frac{\lambda}{p}\,\int_\Omega \frac{|u_n-u|^p}{d_\Omega^{s\,p}}\,dx+o(1),\qquad \mbox{ as }n\to \infty,
	\]
	and
	\[
	[u_n]^p_{W^{s,p}(\mathbb{R}^N)}=[u]^p_{W^{s,p}(\mathbb{R}^N)}+[u_n-u]^p_{W^{s,p}(\mathbb{R}^N)}+o(1),\qquad \mbox{ as } n\to\infty.
	\]
	By inserting these informations in \eqref{minimizzante}, we obtain
	\[
	\mathfrak{F}_{\lambda}(u)+\frac{1}{p}\,[u_n-u]^p_{W^{s,p}(\mathbb{R}^N)}-\frac{\lambda}{p}\,\int_\Omega \frac{|u_n-u|^p}{d_\Omega^{s\,p}}\,dx=m(\lambda)+o(1),\qquad \mbox{ as }n\to \infty.
	\]
	We can now use Hardy's inequality for the function $u_n-u\in \mathcal{X}^{s,p}_0(\Omega;d_\Omega)$. Thanks to the choice of $\lambda$, it holds that
	\[
	\mathfrak{F}_{\lambda}(u)\le m(\lambda)+o(1),\qquad \mbox{ as }n\to \infty,
	\]
	and by taking the limit as $n$ goes to $\infty$, we finally get that $u$ is a minimizer. 
	\par
	By minimality, we get that $u$ must be non-negative. Indeed, by using \eqref{absolutevalue} and observing that 
	\[
	-\int_B u\,dx\ge -\int_B |u|\,dx,
	\] 
	we have 
	\[
	\mathfrak{F}_{\lambda}(u)\ge \mathfrak{F}_{\lambda}(|u|),
	\]
	and the inequality is strict, unless $u\ge 0$ almost everywhere in $\Omega$. Moreover, by minimality $u$ is a weak solution of the Euler-Lagrange equation \eqref{sbilenca},
	as claimed. This in particular proves that $u\not \equiv 0$. Observe that (see Proposition \ref{lm:banach})
	\[
	\mathcal{X}^{s,p}_0(\Omega)\subset W^{s,p}_{\rm loc}(\Omega)\cap L^{p-1}_{s\,p}(\mathbb{R}^N),
	\]
thus $u$ is a local weak supersolution, in the sense of Definition \ref{defi:subsuper}.	
	Finally, by using the minimum principle, we get that $u$ is positive on $\Omega$ (by proceeding as in the proof of \cite[Lemma 3.2]{BBZ}, for example).
\end{proof}

By joining the previous two technical results, we finally get the characterization of the sharp fractional $(s,p)-$Hardy constant stated in Theorem \ref{teo:duale}.
\begin{proof}[Proof of Theorem \ref{teo:duale}]
We first observe that the set of admissible $\lambda$ is non-empty: indeed, it always contains $\lambda=0$. To see this, it is sufficient to observe that any positive constant function is a local weak solution of 
\[
(-\Delta_p)^s u=0,\qquad \mbox{ in }\Omega,
\]
which is \eqref{equazione} for $\lambda=0$.
\par
In order to prove the claimed identity, we first consider the case $\mathfrak{h}_{s,p}(\Omega)=0$. Then, the previous discussion and Lemma \ref{lm:2} imply that the set of admissible $\lambda$ is actually given by the singleton $\{0\}$. 
Thus the conclusion holds.
\par
In the case $\mathfrak{h}_{s,p}(\Omega)>0$,
again by Lemma \ref{lm:2}, we have that $\mathfrak{h}_{s,p}(\Omega)\ge \lambda$ for every $\lambda$ such that \eqref{equazione} admits a positive local weak supersolution. On the other hand, from Lemma \ref{lm:1} we have that for every $\varepsilon>0$ small enough there exists 
	\[
	\mathfrak{h}_{s,p}(\Omega)-\varepsilon<\lambda<\mathfrak{h}_{s,p}(\Omega),
	\]
	such that \eqref{equazione} admits a positive local weak supersolution. This concludes the proof.
\end{proof}


\begin{thebibliography}{100}
	
	\bibitem{An} A. Ancona, On strong barriers and inequality of Hardy for domains in $\mathbb{R}^n$, J. London Math. Soc., {\bf 34} (1986), 274--290.
	
	\bibitem{BBZ} F. Bianchi, L. Brasco, A. C. Zagati, On Hardy's inequality in Sobolev-Slobodecki\u{\i} spaces, preprint (2022), available at {\tt https://cvgmt.sns.it/paper/5711/}
	
	
	
	
	\bibitem{BF} L. Brasco, G. Franzina, Convexity properties of Dirichlet integrals and Picone-type inequalities, Kodai Math. J., {\bf 37} (2014), 769--799.
	
	
	
	
	
	
	\bibitem{BSY} L. Brasco, M. Squassina, Y. Yang, Global compactness results for nonlocal problems, Discrete Contin. Dyn. Syst. Ser. S, {\bf 11} (2018), 391--424. 
	
	\bibitem{BL} H. Br\'ezis, E. Lieb, A relation between pointwise convergence of functions and convergence of functionals,
	Proc. Amer. Math. Soc., {\bf 88} (1983), 486--490.
	
	
	
	
	\bibitem{Dan} D. Daners, ``Domain Perturbation for Linear and Semi-Linear Boundary Value Problems'' in {\it Handbook of differential equations: stationary partial differential equations}. Vol. VI, 1--81. Elsevier/North-Holland, Amsterdam, 2008.
	
	
	
	\bibitem{DPV} E. Di Nezza, G. Palatucci, E. Valdinoci, Hitchhikers guide to the fractional Sobolev spaces, Bull. Sci. Math., {\bf 136}, 521--573.
	
	
	
	
	
	
	
	
	
	
	\bibitem{Fitz} P. J. Fitzsimmons, Hardy's inequality for Dirichlet forms, J. Math. Anal. Appl., {\bf 250} (2000), 548--560.
	
	
	\bibitem{FSspace} R. L. Frank, R. Seiringer, Non-linear ground state representations and sharp Hardy inequalities, J. Funct. Anal., {\bf 255} (2008), 3407--3430. 
	
	
	
	
	
	
	
	\bibitem{KK} J. Kinnunen, R. Korte, Characterizations for the Hardy inequality, {\it Around the research of Vladimir Maz'ya}. I, 239--254,
	Int. Math. Ser. (N. Y.), {\bf 11}, Springer, New York, 2010.
	
	

\bibitem{Leo}  G. Leoni, {\it A first course in Sobolev spaces.} Second edition. Graduate Studies in Mathematics, {\bf 181}. American Mathematical Society, Providence, RI, 2017
	
	\bibitem{LL} E. H. Lieb, M. Loss, {\it Analysis}. Second edition. Graduate Studies in Mathematics, {\bf 14}. American Mathematical Society, Providence, RI, 2001.
	
	
	
	
	
	
	
	
	
	
	
	
	\bibitem{To} G. Tomaselli, A class of inequalities, Boll. Un. Mat. Ital. (4), {\bf 2} (1969), 622--631.
	
	
\end{thebibliography}
\end{document}